\documentclass{amsart}

\usepackage{amssymb}
\usepackage{amsmath}

\newtheorem{theorem}{Theorem}[section]
\newtheorem{lemma}[theorem]{Lemma}
\newtheorem{proposition}[theorem]{Proposition}
\newtheorem{corollary}[theorem]{Corollary}
\theoremstyle{definition}

\theoremstyle{remark}
\newtheorem{remark}[theorem]{Remark}

\numberwithin{equation}{section}

\newcommand{\R}{\mathbb{R}}
\newcommand{\N}{\mathbb{N}}

\newcommand{\K}{\mathfrak{K}}

\renewcommand{\L}{\mathcal{L}}

\renewcommand{\H}{\mathcal{H}}

\author{Mark Allen}
\address{Department of Mathematics, 
Brigham Young University, Provo,
  UT 84602}
\email{allen@mathematics.byu.edu}

\author{Mariana Smit Vega Garcia}
\address{BH 230 Western Washington University, Department of Mathematics, Bellingham, WA 98225}
\email{Mariana.SmitVegaGarcia@wwu.edu}

\title[Thin Free Boundary]{The Fractional Unstable Obstacle Problem}

\begin{document}

\begin{abstract}
 We study a model for combustion on a boundary. Specifically, we study certain generalized solutions of
 the equation 
  \[
   (-\Delta)^s u = \chi_{\{u>c\}} 
  \]
 for $0<s<1$ and an arbitrary constant $c$. Our main object of study is the free boundary $\partial \{u>c\}$. We study the  behavior of the free boundary and 
  prove an upper bound for the Hausdorff dimension of the singular set. We also show that when $s\leq 1/2$ certain symmetric solutions are stable; however, when $s>1/2$ 
  these solutions are not stable and therefore not minimizers of the corresponding
  functional. 
\end{abstract}

\maketitle
\section{Introduction}    \label{s:introduction}
 In this paper we study solutions to an equation that models a boundary reaction. This boundary
 reaction differs from an interior reaction for combustion which has the mathematical model
  \begin{equation} \label{e:tcombust}
   \partial_t u- \Delta u = \chi_{\{u>0\}}.
  \end{equation}
 The authors in \cite{MW07} study traveling waves of \eqref{e:tcombust} by studying the stationary equation
  \begin{equation}  \label{e:combust}
   -\Delta u = \chi_{\{u>0\}}.
  \end{equation}
 The equation \eqref{e:combust} bears a strong resemblence to the obstacle problem which differs from \eqref{e:combust}
 by having a positive sign on the left handside. As noted in \cite{MW07} the ne\-gative sign changes
 the equation to an unstable problem which results in significant differences. 
 
 To formulate the fractional/thin unstable obstacle problem we first fix a bounded domain $U$ in 
 $R^{N}$. The solution should be of the form
  \begin{equation}   \label{e:frac1}
   (-\Delta)^s u = \chi_{\{u>c\}},
  \end{equation}
 where the fractional Laplacian is defined through the spectral decomposition with zero dirichlet boundary data. 
 One of the main difficulties in studying the fractional Laplacian is the nonlocal nature of the operator. By using an extension, we localize the problem and  formulate an even more general problem 
 than \eqref{e:frac1}. 
 By adding an extra variable, the fractional Laplacian in $\R^N$ can be localized by 
 viewing $(-\Delta)^s$ as a Dirichlet to Neumann boundary map (see \cite{CS07}). For a bounded domain the 
 extension is given as follows \cite{st10}:
 
 Let $f$ be a function on $U$ with $f\equiv 0$ on $\partial U$ and $n=N+1$. We consider the domain $U \times \R^+$, and write $(x',x_n)\in \R^{n}$ with $x' \in \R^{n-1}$ and $x_n \in \R$. Let $F$ solve 
   \begin{gather*}
    \text{div}(x_n^a \nabla F(x',x_n)) = 0  \text{ in }  U \times \R \\
                  F(x',0) = f(x')  \\
                  \lim_{x_n \to \infty} F(x',x_n) =0.
   \end{gather*} 
  
 Then 
  \[
   (- \Delta)^s f(x) = c_{N,a} \lim_{x_n \to 0} x_n^{a} \partial_{x_n} F(x',x_n) 
  \]
 where $c_{N,a}$ is a negative constant depending on dimension $N=n-1$ and $a$, were $s$ and $a$ are related 
 by $2s=1-a$. 
 
 In \cite{ALP} the first author studied the two-phase fractional obstacle problem which studies minimizers
 of the functional
  \begin{equation}  \label{e:thinobs}
   \int_{\Omega^+}{|\nabla v|^2 |x_n|^a} + 2\int_{\Omega'}{\lambda_+ v^+ + \lambda_- v^-}
  \end{equation}
 where $\Omega \subset \R^{n}$, $(x',x_n) \in \R^{n-1} \times \R$ are the variables in $\R^{n}$ and 
 $\Omega':= \Omega \cap (\R^{n-1} \times \{0\})$. It is assumed that $\lambda_{\pm} \geq 0$, and $v^+,v^-$ represent the positive and 
 negative parts of $v$. 
 To formulate the thin or fractional unstable obstacle problem we fix a
 bounded smooth domain in $\R^{n}$ 
 We study minimizers of
  \begin{equation} \label{e:twofractional}
   J_a(v,\lambda_+,\lambda_-)
      :=\int_{\Omega^+}{|\nabla v|^2 |x_n|^a} - 2\int_{\Omega'}{(\lambda_+ v^+ + \lambda_- v^-) \ d \H^{n-1}}
  \end{equation}
 with $\lambda_{\pm} \geq 0$. The minimization occurs over the class of functions $H^1(a, \Omega^+)$ (as defined in  
 Section \ref{s:not}) with fixed boundary values on $\partial \Omega \cap \{x_n>0\}$.

 This next propostion illustrates how
 minimizing \eqref{e:twofractional} is always a ``two-phase'' problem even if $\lambda_+ =0$ or $\lambda_- =0$. 
  \begin{proposition}  \label{p:onephase}
   $u$ is a minimizer of $J_a(v,\lambda_+,\lambda_-)$ 
   if and only if $u+cx_n^{1-a}$ is a minimizer of 
   $J_a(w,\lambda_+ -c(1-a), \lambda_- + c(1-a))$ for any constant $c$ such that 
   $-\lambda_- \leq c(1-a) \leq \lambda_+$ and where all test functions are such that 
   $v+cx_n^{1-a}=w$ on $(\partial \Omega)^+$.
  \end{proposition} 
  
  \begin{proof}
   Let $v \in H^1(a,\Omega)$ with prescribed boundary values on $\partial \Omega$. Let $w=v+cx_n^{1-a}$. Then
    \[
     \begin{aligned}
      & J_a(w,\lambda_+ -c(1-a), \lambda_- + c(1-a)) \\
       &\quad = \int_{\Omega^+} \left[ |\nabla v|^2 + 2c(1-a)v_{x_n}x_n^{-a} + c^2(1-a)^2x_n^{-2a})\right]x_n^a \\
        &\qquad -2\int_{\Omega'}{(\lambda_+ -c(1-a))v^+ + (\lambda_- + c(1-a))v^-} \\
       &=\quad  2c\int_{\partial \Omega^+ \cap \{x_n>0\}}v \langle \nu, x_n \rangle \\
        &\qquad+ \int_{\Omega^+}{|\nabla v|^2 |x_n|^a} -2\int_{\Omega'}{\lambda_+ v^+ +  
                   \lambda_- v^- }\\
        &\qquad +2c(1-a)\int_{\Omega'}{v^+ -v^- -v}
     \end{aligned}
    \]
   The first term only depends on the values of $v$ on $\partial \Omega$, and the last term is zero. Therefore, it is 
   clear that  
    \[
     J_a(w,\lambda_+ -c(1-a), \lambda_- + c(1-a)) = J_a(v,\lambda_+,\lambda_-) + C
    \]
   where $C$ is a constant depending on the values of $v$ on $\partial \Omega$. 
  \end{proof}
 
 Because of Proposition \ref{p:onephase} for most of the paper (until Sections \ref{s:stable} and \ref{s:symmetry}) we study minimizers of
 the energy functional 
  \begin{equation}  \label{e:f}
   J_a(v):=\int_{\Omega^+}{ |\nabla u|^2 |x_n|^a} - 2\int_{\Omega'}{u^-}
  \end{equation}
 where $(x',x_n)\in \R^{n}$ with $x' \in \R^{n-1}$ and $x_n \in \R$. 
 It follows from first variation (see Proposition \ref{p:thinsol}) that minimizers are solutions to 
  \begin{equation} \label{e:sol2}
   \int_{\Omega^+}{x_n^a \langle \nabla u, \nabla \psi \rangle } = \int_{\Omega' \cap \{u<0\}}{-\psi}
  \end{equation}
 for every $\psi \in C_0^1(\Omega)$. 
 If $u$ is a minimizer to \eqref{e:f}, then 
 \begin{equation}   \label{e:fsol}
 \begin{aligned}
    \text{div}(x_n^a \nabla u(x',x_n)) &= 0  \text{ in } \Omega^+\\
              \lim_{x_n \to 0} x_n^a u(x',x_n) &= \chi_{\{u(x',0)<0\}}. 
   \end{aligned}
\end{equation}
 From Theorem \ref{t:optimalreg}, a minimizer $u$ will be H\"older continuous on the thin space $\mathbb{R}^{n-1} \times \{0\}$, and so it will be clear 
 that \eqref{e:fsol} will follow from \eqref{e:f} whenever $u(x',0)>0$ or $u(x',0)<0$. From Proposition \ref{p:nosep}, the free boundary $\{u(\cdot, 0)=0\}$ will 
 have $\mathcal{H}^{n-1}$ measure zero, so that \eqref{e:fsol} will hold almost everywhere. 
 
 If $u(x')$ is a solution to \eqref{e:frac1} in $U$, then we add an extra variable $x_n$ so that 
 \begin{equation}   \label{e:fsol1}
 \begin{aligned}
    \text{div}(x_n^a \nabla u(x',x_n)) &= 0  \text{ in } U \times \mathbb{R}^+\\
              \lim_{x_n \to 0} x_n^a u(x',x_n) &= \frac{1}{c_{N,a}}\chi_{\{u(x',0)>c\}}. 
   \end{aligned}
\end{equation}
Solutions to \eqref{e:fsol1} can be found by minimizing 
\[
 \int_{U \times \mathbb{R}^+} |\nabla v(x',x_n)|^2 x_n^a - \frac{2}{-c_{N,a}} \int_{U} (v-c)^+ d \mathcal{H}^{n-1}. 
\]
By adding a constant one may take $c=0$ in \eqref{e:fsol1}. As has already been explained, one may add $c_1 x_n^{1-a}$ to change the multiplicative constant in 
\eqref{e:fsol1}. 
Thus, a more general problem than considering solutions to \eqref{e:frac1} is to consider solutions to \eqref{e:f}. 

 A further result for minimizers is Proposition \ref{p:thinnon}, which gives a nondegeneracy growth condition away 
 from the free boundary. Most of the results in this paper only require a function satisfying \eqref{e:sol2} and 
 the nondegeneracy condition of Proposition \ref{p:thinnon}. 
 
 The weight $x_n^a$ was introduced in \cite{CS07} to study local
 properties of equations involving the fractional Laplacian. The variable $a=1-2s$ where $(-\Delta)^s$ is the
 fractional Laplacian of order $0<s<1$.  The resemblance
 between \eqref{e:frac1} and \eqref{e:combust} gives a mathematical justification for considering minimizers of
 \eqref{e:f} as solutions to the ``fractional unstable obstacle problem''. There is also justification in the 
 applications. When $a=0$ minimizers of \eqref{e:frac1} model temperature control on the boundary. 
 When the sign is negative this corresponds to a reversal of temperature control. In this case more heat is 
 injected when the temperature rises on the boundary. This corresponds to a boundary reaction. 
 
 Just as solutions to the unstable obstacle problem have different properties than the solutions of the
 obstacle problem (\cite{MW07}), minimizers of \eqref{e:f} behave differently than minimizers of \eqref{e:frac1}.
 The greatest difference is the nonseparation of the two phases. One of the main results in \cite{ALP} is the 
 separation of the free boundaries $\Gamma^+ \cap \Gamma^- = \emptyset$ when $a \geq 0$. 
 (See Section \ref{s:not} for a definition of the free
 boundaries.) In stark contrast minimizers of \eqref{e:f} are such that $\Gamma^+ = \Gamma^-$, 
 \ see Theorem \ref{t:thinfree}. 
 
 Solutions to \eqref{e:frac1} also have some differences from solutions to \eqref{e:combust}. When $a\neq0$ solutions
 always achieve the optimal regularity $C^{0,1-a}$ for $a>0$ and $C^{1,-a}$ for $a<0$. For $a=0$ solutions may not 
 have
 the expected Lipschitz regularity whereas minimizers are always $C^{1,1}$ for the unstable obstacle problem.
 
 When $s>1/2 \ (a<0)$, solutions to \eqref{e:sol2} are $C^{1,-a}$, and so from the implicit function theorem the 
 free boundary is a $C^{1,-a}$ manifold wherever the gradient is nonzero. 
 In Section \ref{s:singular} 
 we give our first main result where we prove 
 Theorem \ref{t:haus} which gives an upper bound for the Hausdorff dimension of the points of the free boundary
 where the gradient vanishes. When $s \leq 1/2 \ (a \geq 0)$ the study of the free boundary becomes more difficult 
 because minimizers are not differentiable and have H\"older growth away from the free boundary. In Sections \ref{s:stable} and \ref{s:symmetry}
 we give our second main result which shows that certain symmetric solutions are stable for $s\leq 1/2$, but fail to be stable for $s>1/2$ and hence are not minimizers of 
 the functional.  
 

 The outline of the paper is as follows 
 \begin{itemize}
 \item In Subsection \ref{s:not} we establish the notation to be used throughout the paper.
 \item In Section \ref{s:pre} we discuss existence, regularity, and first variation of minimi\-zers. 
 \item In Section \ref{s:free} we prove topological properties of the free boundary. We also study a class of limiting   
   solutions called ``blow-up'' solutions which are often useful in the study of free boundary problems. 
 \item In Section \ref{s:singular} we prove an upper bound for the Hausdorff dimension of the singular set
        of the free boundary. 
 \item In Section \ref{s:stable} we give a second variational formulation and show certain solutions with singular points are stable for $s\leq 1/2$. 
 \item In Section \ref{s:symmetry} we show that certain symmetric solutions with singular points are not minimizers for $s>1/2$. We also end with a discussion of 
 future directions. 
\end{itemize}

\subsection{Notation}  \label{s:not} 
 The notation for this paper will be as follows. Throughout the paper $2s=1-a$ and $-1<a<1$. $(x',x_n) \in R^n$ with
 $x' \in \R^{n-1}$ and $x_n \in \R$. 
 $\Omega$ will always be a smooth bounded domain that is even with respect to the $x_n$
 variable. 
 \begin{itemize}
 \item $L^2(a, \Omega) := \{f \mid f|y|^{a/2} \in L^2(\Omega)$\}.
 \item  $H^1(a,\Omega) := \{f \mid f, \nabla f \in L^2(a,\Omega)$\}.
 \item  $\Omega':= \{x \in \R^{n-1} \mid (x,0) \in \Omega\}$ 
 \item  $\Omega^+= \{(x',x_n) \in \Omega \mid x_n>0\}$
 \item  $B_r := \{x \in \R^n \mid |x| <1 \}$
 \item $\mathcal{L}_a u := \text{div}(x_n^a \nabla u)$
 \item $f^{\pm}$ denote the positive and negative parts of $x$, respectively, so that $f=f^+ - f^-$. 
 \end{itemize}
 We denote the free boundary as $\Gamma = \Gamma^+ \cup \Gamma^-$ where 
 $\Gamma^+ := \partial \{ u(\ \cdot \ , 0) > 0\}$ and $\Gamma^- := \partial \{ u(\ \cdot \ , 0) < 0\}$.

\section{Preliminaries} \label{s:pre}
 In this section we start by proving existence of minimizers to our functional \eqref{e:f}. To prove existence we state  two notions of trace. The compactness of these trace operators is discussed in \cite{ALP}. 
  \begin{proposition}  \label{p:trace}
   Let $\Omega$ be an open bounded domain with Lipschitz boundary. There exist two compact operators
    \[
     \begin{aligned}
      T_1&: H^1(a,\Omega) \hookrightarrow L^2(a,\partial \Omega)\\
      T_2&: H^1(a,\Omega) \hookrightarrow L^2(\Omega')
     \end{aligned}
    \]
    such that $T_1(\psi) = \psi |_{\partial \Omega}$ and $T_2(\psi)= \psi |_{\Omega'}$ for all $\psi \in C^1(\overline{\Omega})$. 
  \end{proposition}
  
  As a consequence of Proposition \ref{p:trace} we obtain the existence of minimizers. 
  \begin{proposition}  
   Let $\K := \{v \in H^1(a,\Omega^+) \mid v=\phi \text{ on } \partial \Omega \cap \{x_n>0\}\}$. Then there exists $u \in \K$ such that
    \[
     J(u) \leq J(v) 
    \]
   for all $v \in \K$. 
  \end{proposition}
  
  \begin{proof}
   By the boundedness of the trace operator in Proposition \ref{p:trace}
    \[
     \int_{\Omega'}{v^2} \leq C_1^2 \int_{\Omega^+}{x_n^a |\nabla v|^2}.  
    \]
   where $C_1$ is a constant depending on $n,s$, the domain $\Omega$ and the values of $v$ on
   $\partial \Omega$. 
   Also from H\"older's inequality, 
    \[
     \int_{\Omega'}{|v|} \leq |\Omega'|^{1/2} \left( \int_{\Omega'}{v^2} \right)^{1/2}
    \]
   Then
    \[
     \begin{aligned}
      J(v) = \int_{\Omega^+}{x_n^a |\nabla v|^2} - 2\int_{\Omega'}{v_-}   
                                   &\geq \int_{\Omega^+}{x_n^a |\nabla v|^2} - 2|\Omega'|^{1/2} \|v \|_{L^2(\Omega')} \\
                      &\geq \int_{\Omega^+}{x_n^a |\nabla v|^2} - 2|\Omega'|^{1/2}C_1 
                      \left( \int_{\Omega^+}{x_n^a |\nabla v|^2}\right)^{1/2} \\
                      &\geq -C_1^2 |\Omega'|
     \end{aligned}
    \]
   Since $J$ is bounded by below, the existence of a minimizer follows from
   the usual methods of calculus of variations by noting that $\K$ is a closed convex set and using the trace theorem
   from Proposition \ref{p:trace}.  
  \end{proof}
  
  Our functional $J$ satisfies the following rescaling property. 
  \begin{proposition}  \label{p:rescale}
   Let $u$ be a minimizer of $J$ in $B_R^+$. Then $u_r(rx)/r^{1-a}$ is a minimizer of $J$ on $B_{R/r}^+$.  
  \end{proposition}
  
  \subsection{Further Properties}
  Although the negative sign creates significant differences between minimizers of \eqref{e:thinobs} and \eqref{e:f} 
  some initial properties of solutions are the same. This subsection contains preliminary results and proofs that 
  are very similar to those contained in \cite{ALP}. When necessary a modified proof is provided. Such is the case in 
  proving the nondegeneracy results.

  \begin{proposition} \label{p:lattice}
   Let $u_1, u_2$ be minimizers of \eqref{e:f} in $\Omega$ with $u_1 \leq u_2$ on $\partial \Omega^+ \cap \{x_n>0\}$. Then
   $w_1 := \min \{u_1,u_2\}$ and  $w_2 := \max \{u_1,u_2\}$ are minimizers of \eqref{e:f} subject to their respective
   boundary conditions
  \end{proposition}  
  
  
  \begin{corollary}  \label{c:supsol}
   There exists a $\sup$ minimizer $u$ such that $u \geq v$ for all minimizers $v$ satisfying $v \leq u$ on 
   $\partial \Omega^+ \cap \{x_n>0\}$. Furthermore, if $u=c$ a constant on $\partial B_\rho$, then $u$ is radially
   symmetric in the $(x,0)$ variable, so that $u(x,y)=f(|x|,y)$ and $f(r, \ \cdot)$ is nondecreasing in $r$.  
  \end{corollary}
  
 
 This next proposition will enable us to prove a nondegeneracy result. 
 	\begin{proposition}  \label{p:nondegen1}
 	 Let $u$ be a minimizer of \eqref{e:f}. There exists $\epsilon>0$ depending on $n,s$ such that if 
 	 $u\leq \epsilon$ on $\partial B_1^+ \cap \{x_n>0\}$, then $u(0)<0$. 
 	\end{proposition}
 	
 	\begin{proof}
 	 If $u \leq \epsilon$, then $u \leq v$ where $v$ is the $\sup$
 	 solution given in Corollary \ref{c:supsol} with $v \equiv \epsilon$ on $\partial B_1$. 
 	 We will show that $v(0)<0$ for $\epsilon$ small enough. If $v(x,0) \geq 0$, for all $x$,
 	 then 
 	  \[
 	   J(v) = \int_{B_1^+}{|y|^a |\nabla v|^2} \leq J(\epsilon) =0.
 	  \]
 	 where $\epsilon(x,y) \equiv \epsilon$ on all of $B_1^+$ (the constant function). 
 	 Then $v \equiv \epsilon$ and $J(v)=0$. 
 	 Now let $w$ be a candidate such that $w<0$ in $B_1$ and $w\equiv 0$ on $\partial B_1$. 
 	 Notice that 
 	  \[
 	   J(Mw)= M \left(M\int_{B_1^+}{x_n^a |\nabla w|^2} - 2\int_{B_{1}'}{w^-}  \right)
 	  \]
 	 so that there exists $M$ small enough such that $J(Mw)<0$. 
 	 Then for $\epsilon$ small enough,
 	 $J(w+\epsilon)<0=J(\epsilon)$ for $\epsilon$ small.    
 	\end{proof} 
 	The rescaling property of Proposition \ref{p:rescale} combined with Proposition \ref{p:nondegen1} gives
 	\begin{corollary}  \label{c:nondegen}
 	 Let $u$ be a minimizer of \eqref{e:f} in $B_R(x_0,0)^+$ with $u(x_0,0)=0$. Then
 	  \[
 	   \sup_{B_{r}^+(x_0,0)} u \geq Cr^{1-a} \text{  for every  } r<R
 	  \]
 	 where $C$ is a constant depending only on dimension $n$ and $s$. 
 	\end{corollary}
 	
 	\begin{remark}
 	 The $\sup$ can be taken over $\partial B_r^+ \cap \{x_n>0\}$ in Corollary \ref{c:nondegen} since minimizer of \eqref{e:f}
 	 are solutions to \eqref{e:sol2} (see upcoming Proposotion \ref{p:thinsol}) and hence are $a$-subharmonic and therefore satisfy the maximum principle \cite{FKJ}.
 	\end{remark}
 	
 	\begin{remark}
 	 The above result gives nondegeneracy in the full domain $\Omega^+$. A stronger nondegeneracy result for the thin space
 	 $\Omega'$ is given later in Proposition \ref{p:thinnon}. 
 	\end{remark}
 
 This next proposition shows the weighted boundary derivative on $\Omega'$ for $u$ will be constant in a measure theoretic sense whenever ${u<0}$. 
 \begin{proposition} \label{p:thinsol}
   Let $u$ be a minimizer of \eqref{e:f} in $\Omega$. Then for every $\psi \in C_0^2(\Omega)$
    \begin{equation}  \label{e:var2}
     \int_{\Omega^+}{x_n^a \langle \nabla u, \nabla \psi \rangle} = 
      \int_{\Omega' \cap \{u<0\}}{\psi}
    \end{equation}
  \end{proposition}
 	
 	\begin{proof}
 	 Let $\epsilon >0$ and $\psi \in C_0^2(\Omega)$. Then 
 	  \[
 	   \int_{B_1^+}{x_n^a |\nabla (u+\epsilon \psi)|^2} - 2\int_{B_{1}'}{(u+\epsilon \psi)^-}
 	    \geq \int_{B_1^+}{x_n^a |\nabla u|^2} - 2\int_{B_{1}'}{u^-}.
 	  \]
 	 so that 
 	  \[
 	   \begin{aligned}
 	    \int_{B_1^+}{x_n^a \langle \nabla u , \nabla \psi \rangle} 
 	     &\geq \lim_{\epsilon \to 0} \int_{B_{1}'}{\frac{-(u+\epsilon\psi)+u}{\epsilon}
 	      \chi_{\{u+\epsilon\psi <0\}} \chi_{\{u < 0\}}} \\
 	     & \quad +  \int_{B_{1}'}{\frac{-(u+\epsilon\psi)}{\epsilon}
 	      \chi_{\{u+\epsilon\psi <0\}} \chi_{\{u \geq 0\}}} \\
 	     & \quad +  \int_{B_{1}'}{\frac{u}{\epsilon}
 	      \chi_{\{u+\epsilon\psi \geq 0\}} \chi_{\{u < 0\}}}. 
 	   \end{aligned}
 	  \]
 	 Then
 	  \[
 	   \begin{aligned}
 	    -\int_{B_1^+}{x_n^a \langle \nabla u , \nabla \psi \rangle} 
 	     &\leq \lim_{\epsilon \to 0} \int_{B_{1}'}{\psi
 	      \chi_{\{u+\epsilon\psi <0\}} \chi_{\{u < 0\}}} 
 	      +  \int_{B_{1}'}{\frac{-u}{\epsilon} 
 	      \chi_{\{u+\epsilon\psi \geq 0\}} \chi_{\{u < 0\}}}\\
 	     &\leq  \int_{B_{1}'}{\psi
 	      \chi_{\{u+\epsilon\psi <0\}} \chi_{\{u < 0\}}} 
 	      +  \int_{B_{1}'}{\psi
 	      \chi_{\{u+\epsilon\psi \geq 0\}} \chi_{\{u < 0\}}}\\
 	     &=\int_{B_{1}'\cap \{ u<0 \} }{\psi}.
 	   \end{aligned}
 	  \]
 	 Notice that while we assumed $\epsilon>0$, we made no assumption on the sign of $\psi$. Therefore, by 
 	 substituting $\phi = -\psi$ we may conclude equality in the above inequality, and the proof is 
 	 finished. 
 	\end{proof}
 
 The equality \eqref{e:var2} implies the following regularity result for $a \neq 0$ (see \cite{ALP}). 
  \begin{theorem}  \label{t:optimalreg} 
   Let $u$ be a minimizer with $a \neq 0$. If $0<s<1/2$, then $u \in C^{0,1-a}(\Omega^+ \cup \Omega')$ . 
   If $1/2 <s<1$, then 
   $u \in C^{1,-a}(\Omega^+\cup \Omega')$. The interior bounds up to the thin space are given by
    \[
     \begin{aligned}
      \|u \|_{C^{0,1-a}}(\overline{B^+_{r/2}}) \leq C \| u \|_{L^2(a,B_r^+)} \\
      \|u \|_{C^{1,-a}}(\overline{B^+_{r/2}}) \leq C \| u \|_{L^2(a,B_r)} \\
     \end{aligned}
    \]
   where $C$ is a constant depending only on $n$ and $a$.  
   If $s=0$, then $u \in C^{0,\alpha}(B_{1/2}^+ \cup B_{1/2}')$ for every $\alpha<1$.
   Furthermore,
    \[
     \| u \|_{C^{0,\alpha}(\overline{B_{1/2}^+})} \leq C \| u \|_{L^2(B_1)}
    \]
   Where $C=C(n,\alpha)$ is a constant depending on dimension $n$ and $\alpha$.
  \end{theorem}
 When $a \neq 0$ it was necessary to use that the solutions in \cite{ALP} were such that $u^{\pm}$ are both subharmonic
 so that the ACF monotonicity formula could be utilized. Notice that minimizers of \eqref{e:f} are such that $u^-$ is 
 superharmonic. Therefore, we may not expect Lipschitz regularity when $a=0$. 
 This next proposition will transfer
 the H\"older regularity from the thin space to the thick space when $a \neq 0$. 
 
 Monotonicty formulas are extremely useful in proving that so called ``blow-ups'' are homogeneous.
 We write down here a Weiss-type monotonicity formula, whose proof is the same as given in \cite{ALP}.
  \begin{proposition} \label{p:weiss}
   Let $u$ be a minimizer. Then the functional
    \[
     \begin{aligned}
      W(r)=W(r,u)&:= \frac{1}{r^{n-a}} \int_{B_r^+}{x_n^a |\nabla u|^2} - 
                                        \frac{2}{r^{n-a}}\int_{B_{r}'}{u^-} \\
                 & \quad - \frac{1-a}{r^{n+1-a}} \int_{(\partial B_r)^+} {x_n^a u^2 }                       
     \end{aligned}
    \]
   is nondecreasing for $0<r<1$. Furthermore,
    \[
     W(r_2)-W(r_1)=\frac{2}{r^{n-a}} 
       \int_{r_1}^{r_2}\int_{(\partial B_r)^+}{x_n^a \left(\frac{(1-a)u}{r} - u_{\nu} \right)^2} \ dr
    \]
   so that $W$ is constant on $[r_1,r_2]$ if and only if $u$ is homogeneous of degree $2s=1-a$ on the ring 
   $r_1<|x|<r_2$. 
  \end{proposition}
  
  \begin{corollary}  \label{c:homogeneous}
   Let $u$ be a minimizer of \eqref{e:f} with $u(x_0,0)=0$. If $a =0$ assume that 
    \begin{equation}  \label{e:notL}
     \sup_{B_r(x_0)} |u| \leq Cr \text{  for every  } r<r_0 \text{  for  some }  r_0.
    \end{equation}
   If $a<0$ assume that $\nabla_x u(x_0)=0$. Then for any sequence $r_k \to 0$, 
   there exists a subsequence such that the rescalings 
    \[
     u_{r_k}(x):= \frac{u(x_0 + r_kx)}{r_k^{2s}}
    \]
   converge to $u_0$ which is a minimizer of \eqref{e:f} in every compact subset $K \Subset \R^{n}$ and
   $u_0$ is homogeneous of degree $2s=1-a$.  
  \end{corollary}

 We state here Almgren's frequency formula \cite{a00} which we will later use when $a=0$. 
  \begin{proposition} \label{p:almgren}
   Let $u \in H^1(B_1)$ solve $\Delta u =0$. Then
    \[
     N(r) := r \frac{\int_{B_r}{|\nabla u|^2}}{\int_{\partial B_r}{ u^2}}
    \]
   is nondecreasing for $r>0$. Furthermore, $N(r)$ is constant if and only if $u$ is homogeneous of degree
   $k$ for $k \in \N$, and $N(0+):= \lim_{r \to 0} N(r)$ is the degree of homogeneity of the first
   nonzero term of $u$ given by the power series expansion about $0$. 
  \end{proposition}

  \begin{corollary}  \label{c:almgren}
   Let $\Delta u=0$ in $B_1$ and assume $u(0)=0$ and 
    \[
     \int_{B_1}{|\nabla u|^2} = \int_{\partial B_1}{u^2}.
    \] 
   Then $u$ is homogeneous of degree $1$ and hence a linear function. 
  \end{corollary}
  
  We also have the following convergence result for minimizers. 
    
  \begin{proposition}  \label{p:limitsol}
   Let $u_k$ be a sequence of minimizers to \eqref{e:f} in $\Omega$ with
    \[
     \int_{\Omega}{|y|^a u^2} \leq C
    \]
   Then there exists a subsequence relabeled $u_k$ such that 
    \[
     \begin{aligned}
      u_k \to u_0 \text{  in  } C^{0,\alpha} \text{  for  } \alpha < 2s \text{  if  } a \geq 0 \\
      u_k \rightharpoonup u_0 \text{  in  } H^1(a,K)
     \end{aligned} 
    \]
   $u_0$ is a solution to \eqref{e:f} in $K$ with $K \Subset \Omega$. 
  \end{proposition}

\section{The Free Boundary}  \label{s:free}
 In this section we prove results regarding the topology of the free boundary. 
 We first note that in the proof of Proposition \ref{p:thinsol} 
 if we replace $\chi_{\{u < 0\}}$ with $\chi_{\{u \leq 0\}}$ and $\chi_{\{u \geq 0\}}$
 with $\chi_{\{u > 0\}}$, the integrals remain unchanged which allows us to conclude 
 	\[
 	  \int_{B_{1}' }{\psi \chi_{\{u < 0\}}} = \int_{B_{1}' }{\psi \chi_{\{u \leq 0\}}}.
 	\]
 It then follows that $\H^{n-1}(\{u(\cdot,0)=0\})=0$. We may then immediately conclude

 	\begin{proposition} \label{p:nosep}
 	 Let $u$ be a minimizer to \eqref{e:f}. Then $\{(x,0) \mid u(x,0)=0\}$ has 
 	 no interior point in the topology of $\R^{n-1}$ and $\mathcal{H}^{n-1}(\partial \{u(\cdot,0)>0\} \cup  \partial \{u(\cdot,0)<0\}) =0$. 
 	\end{proposition}
 	
 	
 We remark that Proposition \ref{p:nosep} is what one can expect for the free boundary since it is true in the 
 case $s=1$. This result is further strengthened in 
 Theorem \ref{t:thinfree} where we prove $\Gamma^+ = \Gamma^-$.
 
 We now utilize the Weiss monotonicity formula as well as Almgren's frequency function to classify so called
 blow-up solutions when $a=0$. This adds to the results in Corollary \ref{c:homogeneous}. We define
  \[
   \begin{aligned}
    S(r)&:= \left(r^{1-n}\int_{\partial B_r}{u^2} \right)^{1/2}\\
    T(r)&:= r^{1-n}\int_{B_{r}'}{u^-} 
   \end{aligned}
  \]
 For this next Proposition and its proof we evenly reflect $u$ evenly with respect to the $x_n$ variable across the thin space $\R^{n-1}\times \{0\}$.
  \begin{proposition}  \label{p:classify}
   Let $u$ be a minimizer of \eqref{e:f} with $u(x_0,0)=0$ and $a=0$ and $u$ not satisfying 
   \eqref{e:notL} at $x_0$. Then
    \[
     u_r(x) := \frac{u(rx +x_0)}{S(r)}
    \]
   is bounded in $H^1(B_1)$.
   and every limit solution $u_0$ as $r \to 0$ is a linear function in the $x'$ variable. 
  \end{proposition}
  
  \begin{proof}
   By translation we may assume without loss of generality that $x_0 =0$. Define
   	\[
   	 \begin{aligned}
   	  u_r &:= \frac{u(rx)}{S(r)}\\
   	  \tilde{u}_r &:= \frac{u(rx)}{T(r)}.
   	 \end{aligned}
   	\]
   By the Weiss-type monotonicity formula we have
   	\begin{equation}  \label{e:wbound}
   	 \begin{aligned}
   	  \int_{B_1}{|\nabla u_r|^2} &\leq \frac{r}{S(r)} \int_{B_{1}'}{(u_r)^-} + \frac{r^2}{S^2(r)}W(1,u) + 
   	                                 \int_{\partial B_1}{u_r^2}\\
   	                             &= \frac{rT(r)}{S^2(r)} + \frac{r^2}{S^2(r)}W(1,u) +S^2(r)
   	 \end{aligned}
   	\end{equation}
   We claim that 
   	\[
   	 T(r) \leq CS(r) \text{  for  } r<1
   	\]
   for some constant $C$ depending on $x_0$. 
   If the claim is not true then there exists $r_k \to 0$ such that
    \[
     T(r_k) > k S(r_k).
    \]
   By the Weiss type monotonicity formula we have
    \[
   	 \int_{B_1}{|\nabla \tilde{u}_r|^2} \leq \frac{r}{T(r)} + \frac{r^2}{T^2(r)}W(1,u) + 
   	 \frac{S^2(r)}{T^2(r)}.
   	\]
   Now by $C^{0,\alpha}$ regularity (Theorem \ref{t:optimalreg}) and thick nondegeneracy 
   (Corollary \ref{c:nondegen}) 
   for $r<1$ we have $r \leq C S(r)$ for some 
   constant $C$ depending on $x_0$.  Then for $r_k \to 0$
   each of the right hand terms goes to zero in the above inequality,
   so 
    \[
     \int_{B_1}{|\nabla \tilde{u}_{r_k}|^2} \to 0.
    \]
   However, 
   	\[
   	 \int_{B_{1}'}{( \tilde{u}_{r_k} )^-} = 1.
   	\]
   This is a contradiction to the compactness of the trace operator (Proposition \ref{p:trace}), 
   and so our claim is proven. Then we may 
   rewrite \eqref{e:wbound} as 
   	\begin{equation}  \label{e:wbound2}
   	 \begin{aligned}
   	 \int_{B_1}{|\nabla u_r|^2} &\leq \frac{Cr}{S(r)} + \frac{r^2}{S^2(r)}W(1,u) + S^2(r)\\
   	                            &\leq C^2 + C^2W(1,u) +S^2(r)
   	 \end{aligned}
   	\end{equation}
   This proves that $u_r$ is bounded in $H^1 (B_1)$. 
   
   If $S(r)\leq Cr$ for some constant $C$, 
   then by interior $C^{0,\alpha}$ regularity it follows that $u$ satisfies \eqref{e:notL} at $x_0$
   which is a contradiction. Then necessarily $r/S(r) \to 0$.
   We now notice from \eqref{e:wbound2} that for any subsequence $u_r \rightharpoonup u_0$ in $H^1(B_1)$  
   	\[
   	 \int_{B_1}{|\nabla u_0|^2} \leq \int_{\partial B_1}{u_0^2}.
   	\]
   Now we have for any $\psi \in C_0^1{\Omega}$
    \[
     -\int_{B_1}{\langle \nabla u_r , \nabla \psi \rangle} = 
       \frac{r}{S(r)} \int_{B_1'}{\psi \chi_{\{u_r<0\}}}
    \] 
   Since $r/S(r) \to 0$ it follows that $\Delta u_0 = 0$ in $B_1$. 
   Since $u_0(0)$ we utilize Almgren's frequency function to obtain the reverse inequality and conclude that
   	\[
   	 \int_{B_1}{|\nabla u_0|^2} = \int_{\partial B_1}{u_0^2}.
   	\]
   It follows from Corollary \ref{c:almgren} that $u_0$ is homogeneous of degree 1 in all of $B_1$, so $u_0$ is
   linear. Since $u_0(x',-x_n)=u_0(x',x_n)$, it follows that $u_0$ is independent of $x_n$.  
  \end{proof}
  
  \begin{lemma}  \label{l:u}
   Let $u$ be $a$-harmonic in  $\Omega^+$. Assume $u$ is homogeneous of degree $1-a$ and 
    \begin{equation}  \label{e:cc}
     \int_{\Omega^+}{x_n^a \langle \nabla u , \nabla \psi \rangle} = -c\int_{\Omega'} u,
    \end{equation}
   for every $\psi \in C_0^1(\Omega)$. Then $u \equiv cx_n^{1-a}/(1-a)$.
  \end{lemma}
  
  \begin{proof}
   If we let $v=u - cx_n^{1-a}/(1-a)$ and reflect $v$ evenly across the thin space, then 
   \[
    \int_{\Omega} |x_n|^a \langle \nabla v, \nabla \psi \rangle =0
   \]
   for every $\psi \in C_0^1(\Omega)$ so that  $v$ is $a$-harmonic
   in all of $\Omega$. From \cite{a13}, the only $a$-harmonic functions of degree $1-a$ are up to multiplicative constant $|x_n|/x_n^a$
   which is an odd function.
   Since $v$ is also even, we conclude $v \equiv 0$.
  \end{proof}
 This next result shows that the positive and negative phases of minimizers of \eqref{e:f}
 do not separate just as in the local case \eqref{e:combust} when $s=1$. 
  \begin{theorem}  \label{t:thinfree}
   Let $u$ be a minimizer of \eqref{e:f}. Then 
    \[
     \Gamma^+(u) = \Gamma^-(u).
    \] 
  \end{theorem}
  
  \begin{proof}
   For this proof we evenly reflect our minimizers across the thin space. We first consider
   $a \neq 0$. Let $x_0 \in \Gamma^-$. By translation we may assume $x_0=0$.  
   Suppose now by way of contradiction that $u(x,0) \leq 0$ in $B_{\rho}'$ for $\rho >0$. 
   If $a<0$ and $\nabla_x u(0)\neq 0$ we have an immediate contradiction by the $C^{1,\alpha}$
   regularity since $0$ is a local max in $B_1'$. Therefore for $a<0$ or $a>0$ we may  perform a blow-up
    \[
     u_r :=\frac{u(rx)}{r^{1-a}}
    \]
   and for a subsequence we have $u_r \to u_0$ which is homogeneous of degree $1-a$ by Corollary \ref{c:homogeneous}.
   Since $u$ satisfies \eqref{e:cc} in $B_{\rho}$, then $u_0$ also satisfies \eqref{e:cc}. Lemma \ref{l:u}
   then implies
   $u_0 = x_n^{1-a}/(1-a)$. From Proposition \ref{p:limitsol}, we also have that $u_0$ is a minimizer. But $x_n^{1-a}/(1-a)$
   is not a minimizer since it is nonnegative and hence should be $a$-harmonic across the thin space. This is a contradiction.  
   
   Now let $0 \in \Gamma^+$ and suppose $u(0)\geq 0$ in $B_{\rho}'$ for some $\rho>0$. Then by Proposition
   \ref{p:thinsol}, 
    \[
     \int_{B_{\rho}} |x_n|^a \langle \nabla v, \nabla \psi \rangle =0,
    \]
    so that $u$ is $a$-harmonic in $B_{\rho}$. Again we perform a blow-up and obtain
   $u_0$ (not identically zero) homogeneous of degree $1-a$ with $\L_a u_0=0$. Also 
   $u_0(x',x_n)=u_0(x',-x_n)$. This is a contradiction because the only $a$-harmonic functions of degree $1-a$ are 
   $cx_n^{1-a}$, see \cite{a13}, which is odd. 
   
   Now suppose that $a=0$. If $u$ satisfies \eqref{e:notL}, then we may proceed as before when $a \neq 0 $, and perform a blow-up to obtain a homogeneous degree-1 
   minimizer $u_0$. $u_0$ is not identically zero by Corollary \ref{c:nondegen}.

   If $u(x,0)\leq 0$ for $x \in B_{\rho}'$, then as before we may conclude from Lemma \ref{l:u} that 
   $u_0 \equiv c|x_n|$ which is not a minimizer which is a contradiction.

   If $u(x,0)\geq 0$ for $x \in B_{\rho}'$, then $u_0(x,0)\geq 0$ for $x \in B_{\rho}'$, 
   and since $u_0$ is a minimizer of \eqref{e:f}, we have $\Delta u_0 =0$ in 
   $B_1$. Since $u_0$ is homogeneous of degree 1, then $u_0$ is linear. $u_0$ is also even in the $y$ variable, and 
   $u_0(x,0) \geq 0$. Then $u_0 \equiv 0$ which is a contradiction since our blow-up was not identically zero.

   If $u$ does not satisfy \eqref{e:notL} at the origin, then we consider the rescalings
    \[
     u_r(x):=  \frac{u(rx)}{S(r)}.
    \]
   From Propostion \ref{p:classify} we may pick a subsequence $u_r \to u_0$. Since $u_0$ is linear, even in $x_n$, and not
   identically zero, 
   we obtain an
   immediate contradiction. 
  \end{proof}
 
 
 We now give a nondegeneracy result for the thin space. 
  \begin{proposition}  \label{p:thinnon}
   Let $u$ be a minimizer of \eqref{e:f} in $B_R$. Assume $u(0)=0$. If $a \neq 0$ 
    \[
     \sup_{B_r'} u^+ \ , \ \sup_{B_{r}'} u^- \geq Cr^{1-a} \text{ for } r<R/2
    \]
   The constant $C$ depends on $n,s$. 
   If $a =0$, then 
    \[
     \sup_{B_r'} u^+ \ , \ \sup_{B_{r}'} u^- \geq CS(r) \text{ for } r<R/2
    \]
   The constant $C$ depends on dimension $n,s$ and $\|u\|_{L^2(a,B_R)}$.
  \end{proposition}
 
  \begin{proof}
   We first suppose $a\neq 0$. 
   Suppose by way of contradiction that the proposition is not true. Then there exist a sequence of functions $u_k$ 
   with $r_k \leq 1$ such that 
   	\[
   	 \sup_{B_{r}'} ku_k^+ \leq r^{1-a} \text{   or   } \sup_{B_{r}'} ku_k^- \leq r^{1-a}.
   	\]
   We may extract a subsequence from Proposition \ref{p:limitsol} and Proposition \ref{p:rescale}  such that 
    \[
     u_{r_k} := \frac{u(r_k x)}{r_k^{1-a}}
    \]
   is such that $u_{r_k} \to u_0$ with $r_k \to r_0$ with possibly $r_0 = 0$. We have that $u_0$ is a minimizer with
   $u_0(x,0) \leq 0$ or $u_0(x,0) \geq 0$ in $B_{\rho}$ for $\rho <1$. This 
   contradicts Theorem \ref{t:thinfree}. 
   
   If $a=0$ we apply the same method as above to $u_r = u(rx)/S(r)$ and utilize Theorem \ref{t:thinfree}
   to arrive at a contradiction. 
  \end{proof}

\section{Singular points}   \label{s:singular}
 We begin this section by defining the singular set. For $s>1/2 \ (a<0)$, solutions to \eqref{e:f} 
 are $C^{1,-a}$. From the implicit function theorem $\Gamma \setminus \{\nabla u = 0\}$ is a $C^{1,-a}$ surfaces of co-dimension
 2.  For $s>1/2$ we define the singular set of $u$ by 
  \[
   S_u := \Gamma \cap \{\nabla u =0\}.
  \]
 We write $S$ when the function $u$ is understood. 
 
 When $s\leq1/2$ we utilize Lemma \ref{l:unique} below which shows there is a unique homogeneous 
 (of degree $2s$) solution $g(x_1,x_2)$ to \eqref{e:f}. Define $\check{g}(x_1,x_2,\ldots,x_n)=g(x_1,x_n)$. 
 We let $S_u$ consist of those points $x \in \Gamma$ such that
 if $u_0$ is any blow-up of $u$ at $x$, then $u_0$ is a rotation in the first $n-1$ variables of $\check{g}$. 
 We now classify the free boundary points in dimension two. These consist of a single point. By 
 taking a blow-up sequence we know that a blow-up is homogeneous of degree $2s=1-a$. From nondegeneracy, we know that
 there exist nontrivial 
 positive and negative phases in the blow-up. Since in dimension 2, the thin space is of dimension 1, 
 the free boundary consists of a single point. For $a=0$, it has already been shown that the blow-up is a linear 
 function. 
 
 \begin{lemma}  \label{l:unique}
  Let $a\neq0$ and $n=2$. There exists at most one not identically zero and homogeneous of degree $1-a$ solution to \eqref{e:f}. 
 \end{lemma}
 
 \begin{proof}
  Suppose $u,v$ are two solutions. Then either $u +v$ or $u-v$ is $a$-harmonic. The only $a$-harmonic functions
  that are even and homogeneous of degree $1-a$ are identically zero \cite{a13}.
  So $u \equiv v$ or $u \equiv -v$. 
 \end{proof}
 
 \begin{lemma}   \label{l:2zero}
  Let $a<0$ and $n=2$. The only homogeneous of degree $1-a$ solution to \eqref{e:f} is the identically zero solution.
 \end{lemma}
 
 \begin{proof}
  If $u=r^{2s}f(\theta)$, then since $\L_{a} u=0$ for $y>0$ we have 
   \[
    f''(\theta) -a \frac{\sin \theta}{\cos \theta} + (1-a)f(\theta) =0.
   \]
  By adding or subtracting $cy^{1-a}$ we may assume $f(\pi/2)=0$. The solution to the above equation 
  is not a solution to \eqref{e:sol2} rather it is a solution to the stable fractional obstacle problem \cite{ALP}.  
 \end{proof}
 
 In Section \ref{s:stable} we will construct a solution with a singular point at the origin. 
 In \cite{MW07} it is shown that minimizers with $s=1$ have no singular points. For $s<1$ it is not clear
 if the same result is true. However, for $s>1/2$ we prove 
 Theorem \ref{t:haus} which gives a bound on the Hausdorff dimension of the singular set. To prove Theorem \ref{t:haus}
 we only utilize that minimizers of \eqref{e:f} are solutions to \eqref{e:sol2}, have a nondegenerate growth at free
 boundary points, and that blow-ups are homogeneous of degree $1-a$. This matches the result in \cite{MW07} for $s=1$.
 For $s \leq 1/2$, we have Lemma \ref{d:d2}
 which is a partial result for dimension $n=3$. 
 

 The next two Lemmas are standard in the literature and we omit the proofs. 
This next Lemma is analogous to [Proposition 9.6 in \cite{g84}].
 
 \begin{lemma}   \label{l:dreduct}
  Let $u$ be a minimizer of \eqref{e:f}.  Assume also $u$ is homogeneous of degree $1-a$ 
  in dimension $n$. Let $x_0 \in \partial B_1 \cap \Gamma$. Let $u_0$ be any blow-up of $u$ at $x_0$. 
  If $\nu$ is the direction in $x_0$, then 
   \begin{equation}  \label{e:equal}
    u_0(x+t\nu)=u_0(x) \text{  for any  } x \in \R^n \text{  and  } t\in \R. 
   \end{equation}
 \end{lemma}

  \begin{lemma} \label{l:versus}
   Assume $\nu \in \partial B_{1}'$ and 
    \[
     \frac{\partial u_0}{\partial \nu}(x)=0 \text{  for any  } x \in B_1(0).
    \]
   Then $u$ is a minimizer of \eqref{e:f} in $B_1$ if and only if $u$ is a minimizer of \eqref{e:f} in the
   co-dimension 1 set $\Pi_{\nu} \cap B_1$, where $\Pi_{\nu}$ is any plane orthogonal to the direction $\nu$.   
  \end{lemma}

 
 \begin{lemma} \label{d:d2}
  Let $n=3$ and $s>1/2$, and let $u$ be a minimizer of \eqref{e:f} in $\Omega$. For each compact $K \Subset \Omega$, 
  the singular set $K \cap \Gamma \cap \{\nabla_{x'} u =0\}$ contains at most finitely many points. 
 \end{lemma}
 
 \begin{proof}
  We first assume that $s>1/2$. 
  Suppose that the result is not true. Then there exists a sequence of points $x_k \to x_0$ with
  $\{x_k , x_0\} \in S$. We take a blow-up of $u$ with a fixed center $x_0$,
  and with rescalings $r_k = |x_k - x_0|$.  
  By the $C^{1,\alpha}$ convergence, we obtain a blow-up $u_0$ which is homogeneous of degree $2s=1-a$, and 
  $\partial B_{1}'(0)$ contains a point $\zeta \in S_{u_0}$. We may then take a blow-up 
  of $u_0$ at $\zeta$ to obtain $u_{00}$. Since $\nabla u_0(\zeta)=0$, $u_{00}$ is homogeneous of degree $2s=1-a$. 
  Furthermore, by Lemmas \ref{l:dreduct} and \ref{l:versus}, $u_{00}$ will be a minimizer in every compact set of
  the two-dimensional plane $\R \times \{0\} \times \R$. By Lemma \ref{l:2zero} $u_{00}$ is identically zero, but
  this is a contradiction to the nondegeneracy. 
  
 \end{proof}

 As in \cite{MW07}, we now employ the standard dimension reduction argument of Federer to prove the following
 \begin{theorem}  \label{t:haus}
  Let $u$ be a minimizer of \eqref{e:f} with $s>1/2$. 
  The Hausdorff dimension of the singular set of the free boundary
  $S_{u}$ is less than or equal to $n-3$. 
 \end{theorem}

 \begin{proof}
  Suppose the $s>n-3$, and that $\H^s(S)>0$. By [\cite{g84}, Proposition 11.3] and 
  [\cite{g84}, Lemma 11.5] we may obtain a blow-up $u_0$ at $\H^s$ a.e. point of $S$ such that
  $\H^s(S_{u_0})>0$. We may then apply a blow-up at $H^s$ a.e. point of 
  $S_{u_0}\setminus \{0\}$ to obtain $u_{00}$  and $H^s(S_{u_{00}})>0$. By Lemma \ref{l:dreduct} 
  $u_{00}$ is constant in a direction $\nu$ orthogonal to the $x_{n+1}$ direction. Then
  the restriction of $u_{00}$ to a hyperplane $\Pi$ orthogonal to $\nu$ is a minimizer of \eqref{e:f}
  by Lemma \ref{l:versus}. Also, on $\Pi$ we have $H^{s-1}(S_{u_{00}})>0$. 
  We then 
  repeat the procedure $n-2$ more times to obtain a 
  degree $2s$ homogeneous function $\hat{u}$ in $\R^3$
  with $H^{s-(n-3)}(S)>0$, which is a contradiction to Lemma \ref{d:d2}. 
 \end{proof}

  
 

\section{Local Minimizers and Stable Solutions}  \label{s:stable}
 In this section we provide examples of solutions that are local minimizers but not global minimizers. We also discuss examples of solutions that are stable. 
 We prove that for $s\leq 1/2$ all solutions are locally stable. 
 
 For this section we fix the equation we are studying as well as the functional we are minimizing. For symmetry purposes 
 we will assume $\lambda_+=\lambda_-=1$; however, as explained in the introduction, the results for different 
 values of $\lambda$ can be obtained by adding $c_1 u + c_2 x_n^{1-a}$.

 The specific equation we study is 
  \begin{equation}  \label{e:same}
   \int_{\Omega+}{|y|^a \langle \nabla u, \nabla \psi \rangle } = \int_{\Omega'} \psi(\chi_{\{u>0\}}-\chi_{\{u<0\}}).
  \end{equation}
 Solutions of this equation can be found by minimizing the functional
  \begin{equation}  \label{e:minsame}
   \int_{\Omega^+}{|\nabla v|^2 |x_n|^a} - 2\int_{\Omega'}{v^+ + v^- \ d \H^{n-1}}.
  \end{equation}

 We begin this section by computing the second variation as in \cite{MW07}. 
 
 \begin{lemma}  \label{l:2ndvar}
  Let $n \geq 2$. Let $u$ be a minimizer of \eqref{e:twofractional} with $\lambda_+=\lambda_-=1$.
   Assume $s>1/2$ $($so that $a<0)$. Then for $w \in H_0^1(a,B_r(x_0))$,
   \begin{equation}  \label{e:2ndvar}
    0 \leq \int_{B_r^+(x_0)} |\nabla w|^2 x_n^a - 2\int_{\{u=0\}\cap B_{r}'(x_0)} \frac{w^2}{|\nabla u|} d \H^{n-2}
   \end{equation} 
 \end{lemma} 
  
  \begin{remark}
  This formula may seem strange because $w$ is only evaluated on a set of co-dimension $2$. However, when $-1<a<0$
  sets of Hausdorff dimension $n-2$ may have positive capacity, and consequently,  $w$ will have a trace on such sets. 
  \end{remark}
  
 \begin{proof}
  We begin by considering the modified functional
   \[
    E_{\epsilon}(u):=\int_{B_r^+(x_0)} x_n^a|\nabla u|^2 - 2 \int_{B_{r}'(x_0)} \gamma_{\epsilon}(u) d \H^{n-1}.
   \]
  where $\gamma_{\epsilon}(u)$ is an approximation of $|u|$ such that $\gamma_{\epsilon}''(u) = 1/\epsilon$ if $|u|<\epsilon$ and zero otherwise. 
   Now we label 
  the first variation of $E_{\epsilon}$ as
   \[
    t\delta E_{\epsilon}(u)(w) := \int_{B_{r}^+(x_0)} 2\langle \nabla u , \nabla w \rangle  x_n^a - 2\int_{B_{r}'(x_0)} \gamma_{\epsilon}'(u)w.
   \]
  We then have 
   \[
    \frac{1}{t^2}\left(E_{\epsilon}(u+tw)-E_{\epsilon}(u)-t\delta E_{\epsilon}(u)(w) \right)=A_{\epsilon}^t
   \]
  where
   \[
    A_{\epsilon}^t :=\frac{1}{t^2} \int_{B_r^+(x_0)}t^2|\nabla w|^2 x_n^a -\frac{2}{t^2}\int_{B_r'(x_0)} 
       {\gamma_{\epsilon}(u+tw)-\gamma_{\epsilon}(u)-t\gamma_{\epsilon}'(u)w} \ d \H^{n-1}.
   \]
  We can rewrite the second term to obtain
   \[
    \begin{aligned}
    A_{\epsilon}^t &= \int_{B_r^+(x_0)}|\nabla w|^2 x_n^a  
       -2\int_{B_{r}'(x_0)}\int_0^1 \int_0^{\alpha} \gamma_{\epsilon}''(u+\tau tw)w^2 \ d\tau \ d\alpha \ d\H^{n-1}\\
       &= \int_{B_r^+(x_0)}|\nabla w|^2 x_n^a  
       -2\int_0^1 \int_0^{\alpha} \frac{1}{\epsilon}\int_{B_{r}'(x_0)\cap \{|u+\tau tw|<\epsilon\}}w^2 \ d\H^{n-1}\ d\tau \ d\alpha \\
       &= \int_{B_r^+(x_0)}|\nabla w|^2 x_n^a  \\
      &\quad -2\int_0^1 \int_0^{\alpha} \frac{1}{\epsilon}\int_{-\epsilon}^{\epsilon}\int_{B_{r}'(x_0)\cap \{u+\tau tw=\sigma\}}\frac{w^2}{|\nabla(u+ \tau tw)|}  
       \ d\H^{n-2} \ d \sigma\ d\tau \ d\alpha \\
    \end{aligned}
   \]
   with the last equality coming from the coarea formula which may be utilized since $w\equiv 0$ in a neighborhood of the origin. 
  Then as $\epsilon \to 0$
   \[
    A_{\epsilon}^t \to A_0^t := \int_{B_r^+(x_0)}|\nabla w|^2 x_n^a  
       -4\int_0^1 \int_0^{\alpha} \int_{B_{r}'(x_0)\cap \{u+\tau tw=0\}}\frac{w^2}{|\nabla(u+\tau tw)|} \ d\H^{n-2}\ d\tau \ d\alpha.
   \]
  We now conclude that since $u$ is a minimizer
   \[
    0 \leq \frac{1}{t^2}\left(E(u+tw)-E(u)-t\delta E(u)w \right) =A_0^t. 
   \]
  Letting $t \to 0$ we obtain 
   \[
    0 \leq \int_{B_r(x_0)^+} |\nabla w|^2 x_n^a - 2\int_{\{u=0\}\cap B_{r}'(x_0)} \frac{w^2}{|\nabla u|} d \H^{n-2}.
   \]
 \end{proof}
 
 Although this second variation inequality \eqref{e:2ndvar} can be computed as in \cite{MW07}, we cannot utilize this inequality in an elementary manner as in 
 \cite{MW07} to conclude that the singular set for minimizers is empty. For the blow-up of the solution with singular point (constructed shortly below) 
 our inequality \eqref{e:2ndvar}
 scales correctly on both terms so that a scaling argument does not lead to a contradiction. Also, for this blow-up solution, the second term in
 \eqref{e:2ndvar} is finite. This was not the case when $s=1$ which led to the conclusion that the singular set is empty for minimizers when $s=1$.

 We now construct solutions which are not global minimizers, but are local minimizers. For simplicity we give our examples when $n=2$ which also provide examples in higher
 dimensions simply by adding variables. 
 
 We  define 
 \begin{equation}  \label{e:candidate}
  u(x):= c_{a}\int_{B_1'}\frac{w(y)}{|x-y|^{1+a}} \ dy,
 \end{equation}
 with the function 
 \[
  w(r,\theta) = 
  \begin{cases}
   1 &\text{ if } \theta \in [0,\pi/2] \cup [\pi,3\pi/2] \\
   -1 &\text{ if }  \theta \in (\pi/2,\pi) \cup (3\pi/2,2\pi). 
  \end{cases}
 \]
 From the symmetry, it is clear that 
 \[
  w(r,\theta) = 
  \begin{cases}
   u(r,\theta,0) >0  &\text{ if } \theta \in (0,\pi/2) \cup (\pi,3\pi/2) \\
   u(r,\theta,0) <0 &\text{ if } \theta \in (\pi/2,\pi) \cup (3\pi/2,2\pi) \\
   u(r,\theta,0) =0 &\text{ if } \theta \in \{0,\pi/2,\pi,3\pi/2\}. 
  \end{cases}
 \]
The constant $c_a$ is then chosen so that $u$ is a solution to \eqref{e:same} on $B_1^+$.  
 
 We define a local minimizer 
 at $x_0$ to be a 
 solution $u$ such that there exists $r>0$ such that for every $w \in C_{0}^{\infty}(B_r(x_0))$
  \[
   E(u)\leq E(w).
  \]

  \begin{lemma}
   Let $s\leq 1/2$. Let $u$ be as constructed above. 
   Then $u$ is stable on $B_r(x_0)$ for any $x_0 \in B_1$ with $B_r(x_0)\subset B_1$. 
  \end{lemma}
 
  \begin{proof}
  Let $w \in C_0^{\infty}(B_r \setminus \{0\})$. As before we compute
    \[
     \int_{B_r^+(x_0)}|\nabla w|^2 x_n^a  
       -2\int_0^1 \int_0^{\alpha} \frac{1}{\epsilon}\int_{B_{r}'(x_0)\cap \{|u+\tau tw|<\epsilon\}}w^2 \ d \H^{n-1} \ d\tau \ d\alpha \\
    \]
   Since for $s<1/2$, $|\nabla u|\to \infty$ as we approach the set $\{u=0\}$, for $t$ small enough $|\nabla u + \tau tw|\neq 0$. We now show how
   to rigorously apply the co-area formula. On $|u|>\delta$, the solution $u$ is Lipschitz.  Then  
    \[
     \begin{aligned}
      \frac{1}{\epsilon}\int_{B_{r}'(x_0)\cap \{|u+\tau tw|<\epsilon\}}w^2  d\H^{n-1}
       &= \lim_{\delta \to 0} \frac{1}{\epsilon}\int_{B_{r}'(x_0)\cap \{|u+\tau tw|<\epsilon\}\cap \{|u|>\delta\}}w^2  d\H^{n-1}\\
       &= \lim_{\delta \to 0} \frac{1}{\epsilon}\int_0^{\epsilon}\left(\int_{\{u+\tau tw=\theta\}}\frac{w^2\chi_{|u|>\delta}}{|\nabla u + \tau tw|}  d\H^{n-2} \right) \ d\theta \\
       &=  \frac{1}{\epsilon}\int_0^{\epsilon}\left(\int_{\{u+\tau tw=\theta\}}\frac{w^2}{|\nabla u + \tau tw|}  d\H^{n-2} \right) d\theta.
     \end{aligned}
    \]
   Then 
    \[
     \lim_{\epsilon \to 0}\frac{1}{\epsilon}\int_0^{\epsilon}\left(\int_{\{u+\tau tw=\theta\}}\frac{w^2}{|\nabla u + \tau tw|}  d\H^{n-2} \right) d\theta
     = \int_{\{u+ \tau tw=0\}} \frac{w^2}{|\nabla u + \tau tw|} d\H^{n-2}. 
    \]
   So as before we obtain that 
    \[
     A_0^t = \int_{B_r^+}|\nabla w|^2 x_n^a-4\int_0^1 \int_0^{\omega} \int_{B_{r}'(x_0)\cap \{u+\tau tw=0\}} \frac{w^2}{|\nabla u + \tau tw|} d\H^{n-2} d\tau d\omega. 
    \]
   For fixed $w$, as $t\to 0$, $|\nabla u + \tau tw|\to \infty$, so that for fixed $w$,
    \[
     A_0^t >0 
    \]
   for any $t\leq t_0$ with $t_0$ depending on $w$. 
  \end{proof}

 \section{Nonminimizing symmetric solutions}  \label{s:symmetry}
 If $a<0$ and $u$ is a minimizer of the functional, then the second variational formula is 
  \begin{equation}   \label{e:2variation}
     2\int_{\{u=0\}\cap B_{r}'(x_0)} \frac{w^2}{|\nabla u|} d \H^{n-2} \leq \int_{B_r(x_0)^+} |\nabla w|^2 x_n^a,
  \end{equation}
 as long as $w \in H_0^1(a,B_r(x_0))$. 
   
   We now consider the symmetric solution $u$ constructed in Section \ref{s:stable}. We will shortly show 
   that solution  $u$ to \eqref{e:same} is not a minimizer of \eqref{e:minsame}, see Remark \ref{r:notmin}. 
   Even though $u$ is not a minimizer 
   of \eqref{e:minsame}, one may verify that  $u$ in \eqref{e:candidate} 
  satisfies the nondegeneracy condition 
  \begin{equation}   \label{e:nondegene}
   \sup_{B_r} |u| \geq cr^{1-a}. 
  \end{equation}
  Indeed, if $u$ did not satisfy \eqref{e:nondegene}, then if $u_r=r^{a-1}u(rx)$, and if $r \to 0$, then $u_r \to u_0\equiv 0$ in 
  $C^{1,\beta}$ for any $\beta<-a$. However,  one would also have preserved in the limit 
  that $\lim_{x_n \to 0} x_n^a u_{x_n} = \pm 1$ depending on the angle $\theta$ resulting in a contradiction. 
  
  Since \eqref{e:nondegene} holds, by letting $u_r=r^{a-1}u(rx)$ and picking a subsequence $r_k \to 0$, 
  we have that $u_{r_k} \to u_2$ which will 
  be a solution to \eqref{e:same} and also homogeneous of degree $1-a$ by Proposition \ref{p:weiss}. 
   
   \begin{theorem}   \label{t:u2}
    Let $u_2$ be the symmetric solution to \eqref{e:same} described above which is homogeneous of degree $1-a$. 
    Then $u_2$ is not stable at the origin and therefore not a minimizer of \eqref{e:minsame} on any open bounded set 
    containing the origin.    
    \end{theorem}
   
   \begin{remark}  \label{r:notmin}
    Since $u_2$ is not a minimizer of \eqref{e:minsame}, it follows that $u$ as defined in \eqref{e:candidate} is also 
    not a minimizer. Indeed, if $u$ is a minimizer, then $u_r=r^{a-1}u(rx)$ would then also be a minimizer, and so in 
    the limit $u_2$ would also have to be a minimizer. 
    \end{remark}
   
   \begin{proof}
    We define 
    \begin{equation}  \label{e:candidate2}
      v(x):= c_{a}\int_{-1}^0\frac{-2}{|x-(y,0,0)|^{1+a}} \ dy +  c_{a}\int_{0}^1\frac{2}{|x-(y,0,0)|^{1+a}} \ dy,
     \end{equation}
     with $c_a$ as in \eqref{e:candidate}. Notice that $\mathcal{L}_a v=0$ off the $x$-axis.  
          We may write explicitly (we assume $x_1>0$)
     \[
     \begin{aligned}
     v(x_1,0,0)&= c_a \int_0^{2x} \frac{2}{|(x_1,0,0) - (y,0,0)|} \ dy  \\
     &\quad + c_1 \int_{-1}^{-1+x} \frac{-2}{|(x_1,0,0) - (y,0,0)|} \ dy \\
     &= \frac{4c_a}{-a} x^{-a}+ c_1 \int_{-1}^{-1+x} \frac{-2}{|(x_1,0,0) - (y,0,0)|} \ dy
     \end{aligned}
     \]
     If we rescale by $v_r(x)=r^{a}v(rx)$, and let $r \to 0$, then the second term disappears and we obtain $v_2$
     with 
     \[
        \mathcal{L}_a v_2(x_1,x_2,x_3) =
        \begin{cases}
         0 &\text{ if } x_2 \neq 0 \text{ or } x_3 \neq 0 \\
         2 &\text{ if } x_1 >0 \text{ and } x_2=x_3=0 \\
         -2 &\text{ if } x_1 <0 \text{ and } x_2=x_3=0.
        \end{cases}
     \]
     Then $\mathcal{L}_a (\partial_{x_1} u_2 - v_2)=0$ and homogeneous of degree $-a$. From the classification 
     of homogeneous solutions \cite{a13}, there are no homogeneous solutions of degree $-a$, we conclude that 
     $\partial_{x_1}u_2 - v_2 \equiv 0$. Thus we have determined that 
     \[
      \left| \frac{\partial}{\partial x_2} u_2 (x_1,0,0)\right|=
       \frac{4c_a}{-a} x^{-a}. 
     \]
   \end{proof}
   We now construct our test function by defining 
   \[
    w(x)= c_a \int_{0}^1 \frac{1}{|x-(y,0,0)|^{1+a}} \ dy. 
   \]
   We have constructed $w$ so that 
   \[
    \int_{\mathbb{R}_+^3} |\nabla w|^2 x_n^a \ dx = \int_{-1}^1 w(y,0,0) \ dy.  
   \]
   Furthermore, 
   \[
   w(x_1,0,0)= \frac{c_a}{-a}(x^{-a} + (1-x)^{-a}). 
   \]
  Then
   \[
   \begin{aligned}
    &\int_{0}^1 w(y,0,0) \ dy - 2\int_0^1 y^{a} \frac{w^2(y,0,0)}{|\nabla u_2(y,0,0)|} \ dy \\
    &= \frac{c_a}{-a} \int_{0}^1 y^{-a} +(1-y)^{-a} \ dy - \frac{c_a}{-2a}\int_0^1 y^{a} (y^{-a} +(1-y)^{-a})^2 \ dy \\
    &= \frac{c_a}{-a} \frac{2}{1-a} - \frac{c_a}{-2a}\int_0^1 y^{-a} + 2(1-y)^{-a} + y^a(1-y)^{-2a} \ dy \\
    &= \frac{c_a}{-a}\left(\frac{1}{2(1-a)} - \frac{1}{2}\int_0^1 y^a (1-y)^{-2a} \ dy \right) \\
    &= \frac{c_a}{-2a} \left(\frac{1}{1-a} - B(1+a,1-2a) \right) \\
    &= \frac{c_1}{-2a} (B(1-a,1) -B(1-2a,1+a))<0,
    \end{aligned}
   \]
   where $B(1+a,1-2a)$ is the Beta function and a proof of the last inequality is in Lemma \ref{Beta} at the end of the paper.
     
   Then 
   \[
    \begin{aligned}
   &\lim_{r \to \infty} \int_{B_r^+(x_0)} |\nabla w|^2 x_n^a - 2\int_{\{u_2=0\}\cap B_{r}'(x_0)} \frac{w^2}{|\nabla u_2|} d \H^{1} \\
   &\leq  \int_{-1}^1 w(y,0,0) \ dy - 2\int_{0}^1 y^{a} \frac{w^2(y,0,0)}{|\nabla u_2(y,0,0)|} \ dy \\
   &<0.
   \end{aligned}
   \]
   Now $w$ is not a valid test function since it does not have compact support. However, we may approximate $w$ with functions
   that do have compact support so that   \eqref{e:2ndvar} is not true for $u_2$ on $B_r$ for large enough $r$. Then by scaling, the inequality \eqref{e:2ndvar} is not true on any open ball containing the origin. This concludes the proof. 
   
   The solution $u_2$ is not the only symmetric solution to \eqref{e:same} that is homogeneous and symmetric. We now construct
   an infinitely family $u_i$ for $i \geq 2$ and $i \in \mathbb{Z}$. Although our construction for $u_2$ will also be valid for constructing all $u_i$, we choose now a different construction which will be useful later for the comparison principle. 
   We consider the domain $U_i = \{(x_1,x_2,x_3) \in B_1 \mid x_3 \geq 0 \text{ and } 0<\arctan(y/x)< \pi/i\}$. We let 
   $\tilde{u}_i$ be the positive minimizer of \eqref{e:minsame} subject to the boundary condition $\tilde{u}_i=0$ 
   on $\partial U_i \cap \{x_3>0\}$. 
   It is clear that there will be two minimizers $\pm \tilde{u}_i$. We may now use reflection on $\tilde{u}_i$
   and reflect across the plane  $-\sin(\pi/i)x+\cos(\pi/i)y+0z$ and then reflect $2i-1$ more times to obtain that $\tilde{u}_i$
   is a solution to \eqref{e:same} in $B_1^+$. Now \eqref{e:nondegene} will also hold for $\tilde{u}_i$ for the same reasons as for 
   $u$ given by \eqref{e:candidate}. Therefore, we let 
   \begin{equation}   \label{e:candidates}
    u_i = \lim_{r \to 0} r^{a-1} \tilde{u}_i(rx). 
   \end{equation} 
   $u_2$ constructed in this manner is the same as $u_2$ constructed before because by subtracting the two we obtain a 
   solution to $\mathcal{L}_a$ that is homogeneous of degree $1-a$ and may be reflected evenly across $\{x_3=0\}$ and 
   remain a solution. By the classification of homogeneous solutions in \cite{a13} it follows that we obtain the same $u_2$. 
   This argument also explains why we may take $\lim_{r \to 0}$ in \eqref{e:candidates} and not have to worry about taking
   a subsequence.

  \begin{theorem}\label{t:i}
   Let $u_i$ be given by \eqref{e:candidates}. Then $u_i$ is not a stable solution, and  consequently not 
   a minimizer on any bounded open set containing the origin.   
  \end{theorem}
  
  \begin{proof}   
   We note that 
   \[
    \tilde{u}_i \leq \tilde{u}_2 \text{ on }  \partial U_i \cap \{x_3 >0\}. 
   \]
   Then if we reflect $\tilde{u}_2 - \tilde{u}_i$ evenly across $\{x_3=0\}$, then $\mathcal{L}_a (\tilde{u}_2 - \tilde{u}_i)=0$
   in the domain full domain $\tilde{U}_i= \{(x_1,x_2,x_3) \in B_1 \mid 0< \arctan(y/x)< \pi/i \}$. Then by the maximum principle
   we have that $\tilde{u}_2 \geq \tilde{u}_i$ in $\tilde{U}_i$. Then in the limit we also have 
   $u_2 \geq u_i$ in $\tilde{U}_i$. Therefore, 
   \[
    |\partial_{x_2} u_i(x_1,0,0)| \leq |\partial_{x_2} u_2(x_1,0,0)|= \frac{4 c_a}{-a} x_1^{-a}. 
   \] 
   
  Then the same proof as in Theorem \ref{t:u2} will hold. 
  \end{proof}

  The following questions remain open and are of further interest. 
  \begin{itemize}
   \item Does the free boundary $\Gamma$ have higher regularity (for instance $C^{\infty}$ regularity) when $s>1/2$?
   \item What is the regularity of the free boundary when $s\leq 1/2$?
   \item Are the symmetric solutions not only stable when $s \leq 1/2$ but also minimizers of the functional?
   \item Are there any singular points for minimizers when $s >1/2$?
  \end{itemize}
  
  \begin{lemma}\label{Beta}  Let $-1<a<0$. Then 
\[
\frac{1}{1-a}<B(1+a,1-2a).
\]
\end{lemma}

\begin{proof} Recall that $B(x,y)=\frac{\Gamma(x)\Gamma(y)}{\Gamma(x+y)}$ and the identities 
\[
\Gamma(1+z)=\Gamma(z)z, \qquad \Gamma(1-z)\Gamma(z)=\frac{\pi}{\sin(\pi z)}, \qquad B(x,y)B(x+y,1-y)=\frac{\pi}{x\sin(\pi y)}.
\]
Using these identities, one obtains
\begin{align*}
B(1+a,1-2a)&=
\frac{\Gamma(1-(-a))\Gamma(1+(-2a))}{\Gamma(1+(1-a))}
=\frac{-2\pi a}{\sin(-\pi a)(1-a)}\frac{\Gamma(-2a)}{\Gamma(-a)^2(-a)}\\
&=\frac{2\pi}{(1-a)\sin(-\pi a)}\frac{1}{B(-a,-a)}.
\end{align*}
Our problem then reduces to showing that for $0<x<1$, $1<\frac{2\pi}{\sin(\pi x)}\frac{1}{B(x,x)}.$
Since $B(x,x)>0$, it suffices to show that for $0<x<1$, $B(x,x)<\frac{2\pi}{\sin(\pi x)}.$
We will show
\begin{equation}\label{main}
B(x,x)<\frac{2}{x}.
\end{equation}
Equation \eqref{main} is equivalent to $\frac{\Gamma(x)^2}{\Gamma(2x)}<\frac{2}{x}$, that is, $x\Gamma(x)^2<2\Gamma(2x)$. Multiplying by $x$, our problem is equivalent to showing that $x^2\Gamma(x)^2<2x\Gamma(2x)$.
Using the identity $\Gamma(z)z=\Gamma(z+1)$, this in turn reduces to $\Gamma(x+1)^2<\Gamma(2x+1)$. Taking $\log$ on both sides, this is equivalent to $2\log\Gamma(x+1)<\log\Gamma(2x+1)$. Notice that $\Gamma(1)=1$, hence \eqref{main} is equivalent to
\[
2\log\Gamma(x+1)<\log\Gamma(2x+1)+\log\Gamma(1).
\]
Since $\Gamma$ is log-convex and $x+1=\frac{2x+1}{2}+\frac{1}{2}$, we conclude that \eqref{main} holds.

  \end{proof}

\bibliographystyle{amsplain}
\bibliography{refunstable}

\providecommand{\bysame}{\leavevmode\hbox to3em{\hrulefill}\thinspace}
\providecommand{\MR}{\relax\ifhmode\unskip\space\fi MR }
\providecommand{\MRhref}[2]{%
  \href{http://www.ams.org/mathscinet-getitem?mr=#1}{#2}
}
\providecommand{\href}[2]{#2}
\begin{thebibliography}{1}

\bibitem{a13}
Mark Allen, \emph{Thin free boundary problems}, Ph.D. thesis, Purdue
  University, August 2013.

\bibitem{ALP}
Mark Allen, Erik Lindgren, and Arshak Petrosyan, \emph{The two-phase fractional
  obstacle problem}, SIAM J. Math. Anal. \textbf{47} (2015), no.~3, 1879--1905.
  \MR{3348118}

\bibitem{a00}
Frederick~J. Almgren, Jr., \emph{Almgren's big regularity paper}, World
  Scientific Monograph Series in Mathematics, vol.~1, World Scientific
  Publishing Co., Inc., River Edge, NJ, 2000, $Q$-valued functions minimizing
  Dirichlet's integral and the regularity of area-minimizing rectifiable
  currents up to codimension 2, With a preface by Jean E. Taylor and Vladimir
  Scheffer. \MR{1777737 (2003d:49001)}

\bibitem{CS07}
Luis Caffarelli and Luis Silvestre, \emph{An extension problem related to the
  fractional {L}aplacian}, Comm. Partial Differential Equations \textbf{32}
  (2007), no.~7-9, 1245--1260. \MR{2354493 (2009k:35096)}

\bibitem{FKJ}
E.~B. Fabes, C.~E. Kenig, and D.~Jerison, \emph{Boundary behavior of solutions
  to degenerate elliptic equations}, Conference on harmonic analysis in honor
  of {A}ntoni {Z}ygmund, {V}ol. {I}, {II} ({C}hicago, {I}ll., 1981), Wadsworth
  Math. Ser., Wadsworth, Belmont, CA, 1983, pp.~577--589. \MR{730093
  (85m:35028)}

\bibitem{g84}
Enrico Giusti, \emph{Minimal surfaces and functions of bounded variation},
  Monographs in Mathematics, vol.~80, Birkh\"auser Verlag, Basel, 1984.
  \MR{775682 (87a:58041)}

\bibitem{MW07}
R.~Monneau and G.~S. Weiss, \emph{An unstable elliptic free boundary problem
  arising in solid combustion}, Duke Math. J. \textbf{136} (2007), no.~2,
  321--341. \MR{2286633 (2007k:35527)}

\bibitem{psu12}
Arshak Petrosyan, Henrik Shahgholian, and Nina Uraltseva, \emph{Regularity of
  free boundaries in obstacle-type problems}, Graduate Studies in Mathematics,
  vol. 136, American Mathematical Society, Providence, RI, 2012. \MR{2962060}

\bibitem{st10}
Pablo~Ra{\'u}l Stinga and Jos{\'e}~Luis Torrea, \emph{Extension problem and
  {H}arnack's inequality for some fractional operators}, Comm. Partial
  Differential Equations \textbf{35} (2010), no.~11, 2092--2122. \MR{2754080
  (2012c:35456)}

\end{thebibliography}

\end{document}